\documentclass[11pt,reqno]{amsart}
\usepackage[utf8]{inputenc}
\usepackage{amsmath, amsfonts, amsthm, amssymb}
\usepackage[margin=1in]{geometry}
\usepackage{mathtools}
\usepackage{stix}
\usepackage{float}
\usepackage{hyperref}

\newtagform{brackets}[\textbf]{$\langle$}{$\rangle$}
\usetagform{brackets}

\allowdisplaybreaks[4]

\newcommand{\gte}{\geqslant}
\newcommand{\lte}{\leqslant}

\newcommand\fsk[2]{F_{#1,#2}}
\newcommand\pp{\,.}
\newcommand\pc{\,,}
\newcommand\dcs{\,\,}

\newcommand\sms[2]{#1^{(#2)}}
\newcommand{\mmod}[1]{\ (\mathrm{mod}\ #1)}
\newcommand\LastS{200}

\makeatletter
\newcommand{\oset}[3][0ex]{%
  \mathrel{\mathop{#3}\limits^{
    \vbox to#1{\kern-2\ex@
    \hbox{$\scriptstyle#2$}\vss}}}}
\makeatother
\newcommand\eqs[2][0ex]{\oset[#1]{#2}{\sim}}

\newcommand\conj[1]{\oset{#1}{\leftrightarrow}}

\newcommand\mbr{\mathbb R}
\newcommand\mbz{\mathbb Z}
\newcommand\mbn{\mathbb N}
\newcommand\mcm{\mathcal M}
\newcommand\mcf{\mathcal F}
\newcommand\mcp{\mathcal P}

\newcommand\MLAB{\textsc{Matlab}\texttrademark}
\newcommand\SAGE{\textsc{Sage}}

\newcommand\vfs{\vphantom{\binom ns}}
\newcommand{\phn}{\vphantom{$1^{1^1}$}}

\newtheorem{theorem}{Theorem}[section]

\newtheorem{corollary}[theorem]{Corollary}
\newtheorem{prop}[theorem]{Proposition}
\newtheorem{problem}[theorem]{Problem}
\newtheorem{question}[theorem]{Question}
\newtheorem{conjecture}[theorem]{Conjecture}

\newtheorem*{definition}{Definition}
\newtheorem*{notation}{Notation}

\begin{document}

\title[Is the Multiset of $n$ Integers Uniquely Determined by the Multiset of Its $s$-sums?]{Is the Multiset of $n$ Integers Uniquely Determined \\  by the Multiset of Its $s$-sums?}
\author{Dmitri V.~Fomin}
\address{Boston, USA}
\email{fomin@hotmail.com}
\date{\today}
\dedicatory{This paper is dedicated to the memory of Oleg Izhboldin (1963-2000)}

\begin{abstract}
This is a survey of all available information on a remarkable problem in number theory proposed by Leo Moser in 1957. In general form the question is: can a collection of $n$ numbers be uniquely restored given the collection of its $s$-sums? We describe results and techniques from sixty years of research in this area. Some new findings and open questions are presented.
\end{abstract}

\maketitle

\section{Introduction}
\label{sec:Intro}

Sixty years ago, in 1957, American Mathematical Monthly published the following relatively simple problem in number theory proposed by Leo Moser (see \cite{Mos}):

\begin{problem}
\label{th:prob_M2}
(a) The ten numbers $s_1 \lte s_2 \lte \cdots \lte s_{10}$ are the sums of the five unknown numbers $x_1 \lte x_2 \lte \cdots \lte x_5$ taken two at a time. Determine the $x$'s in terms of the $s$'s.

(b) Show that if $s_1 < s_2 < \cdots < s_6$ are six distinct numbers formed by taking the sums of four numbers two at a time, then there exist four other numbers which give the same sums when added in pairs.
\end{problem}

Naturally---after it was quickly solved---the problem was immediately generalized and reformulated. And then it turned out to be quite an interesting little question\dots

\begin{problem}
\label{th:prob_M2G}
Let $A$ be a collection (multiset) of $n$ numbers $a_1\lte a_2\lte\cdots \lte a_n$. Consider the multiset $\sms A2$ of $\binom{n}2$ $2$-sums of multiset $A$, i.e., collection of all sums of the kind $a_{i_1}+a_{i_2}$, where $1\lte i_1 <i_2 \lte n$. Is it possible to restore $A$ given $\sms A2$?
\end{problem}

The numbers here could be complex or even belong to an arbitrary field of characteristic zero. That really doesn't matter as we will learn shortly.

In this generalization the problem asks whether a multiset is uniquely determined by (or can be recovered from) the multiset of its $2$-sums.

It was later presented in the literature (e.g., see \cite{BomBolOne}) using somewhat different terminology. 

\begin{problem}
\label{th:prob_M2N}
A malicious farmer's apprentice was asked to provide the list of weights of $n$ bags of grain. Instead he weighed them two at a time and recorded all $n(n - 1)/2$ combined weights written down in some random order. Is it possible to find the weights of bags (up to permutation of bags)?
\end{problem}

By the way, it seems that the apprentice was not only malicious but also somewhat dense---instead of performing only $n$ weighings he did a whole lot more of them.

Again the problem is posed as a ``recovery'' question---whether an unknown multiset can be uniquely restored from the multiset of its pairwise sums. 

Now, each interesting question, theorem or conjecture deserves a nice name that easily rolls off the tongue in lectures and discussions. ``\textit{Fermat's Last Theorem}'', ``\textit{Riemann Hypothesis}'', ``$P=NP$'', ``\textit{Collatz $3k+1$ Conjecture}''---all these names are short and to the point. I submit that ``\textit{Multiset Recovery Problem}'' sounds just as neat while describing the issue with decent precision.

The original problem \textbf{\ref{th:prob_M2}} was indeed quite easy. However, its generalization \textbf{\ref{th:prob_M2G}} was not. Still, it did not present a serious obstacle; the answer was quickly discovered and so the problem was generalized even further.

\begin{notation}
For any pair of positive integers $n$ and $s$ such that $n\gte s$ we will denote by $\sms As$ the multiset of $s$-sums of $A$, i.e. collection of all sums of the kind
$$
a_{i_1}+a_{i_2}+\ldots+a_{i_s}\pc
$$
where $1\lte i_1 < i_2 < \ldots < i_s \lte n$.
\end{notation}

\begin{problem}
\label{th:prob_MNK} Consider positive integers $n$ and $s$ with $n > s$. Do there exist two distinct $n$-multisets $A$ and $B$ such that $\sms As = \sms Bs$?
\end{problem}

Such a pair of multisets would represent a "recovery failure". Indeed, in this case, given multiset $M = \sms As = \sms Bs$ it is impossible to determine the original multiset.

\begin{definition}
If two multisets $A$ and $B$ have the same collections of $s$-sums---that is, if $\sms As = \sms Bs$ ---then we will call these multisets \textbf{$s$-equivalent} (or when this will not cause any confusion, simply equivalent) and this relation will be denoted as $A \eqs s B$ (or simply $A\sim B$).
\end{definition}

\begin{definition}
We will call a pair of natural numbers $(n,s)$ \textbf{singular} if it represents a nontrivial ``multiset recovery failure''---i.e., $n>s$ and there exist two different $s$-equivalent $n$-multisets $A$ and $B$ ($A\neq B\dcs\&\dcs A \eqs s B$).
\end{definition}

So all the above problems can be reworded as questions about singular pairs. The ultimate goal is to describe those pairs in some easily ``computable'' way.

\begin{notation}
For any natural number $s$ by $\mcm_s$ we will denote the set of all natural numbers $n>s$ such that pair $(n,s)$ is singular.
\end{notation}

For instance, $\mcm_1$ is obviously empty. Also \textbf{Question \ref{th:prob_M2}} could be reformulated as follows: does $\mcm_2$ contain numbers 5 and 4? (Actually, as far as the second half of that question goes, this is not an entirely precise reformulation, but let's not nitpick).

\smallskip

This article will present all currently known results on the Multiset Recovery Problem and the methods involved. We will also discuss some new facts and conjectures.

\section{Historical timeline}
\label{sec:Timeline}

\subsection*{(1957)}

Just a few words about Moser's possible motivation for this problem. A few years before that, in 1954, Leo Moser and Jim Lambek published article \cite{LamMos} about pairs of complementary subsets of $\mbn$ where they have proved \textbf{Lambek-Moser Theorem} about partitions of $\mbn$ and how they are related to sequences of numbers of the form $\{f(n)+n\}$ where $f:\mbn\to\mbz_{\gte0}$ is some arbitrary nondecreasing unbounded function.

This investigation seems to be quite close to questions about how finite or infinite sets of natural numbers overlap or complement each other when being translated.

I suspect that this was how Leo Moser stumbled upon questions about multiset recovery for the case of $s=2$---but of course, this is pure speculation. However, in their later short article \cite{LamMos2} Lambek and Moser mention both Multiset Recovery Problem and complementary sequences of integers literally on the same page.

\subsection*{(1958)}

In 1958, almost immediately after Moser has posed his original question, Selfridge and Straus in their article \cite{SelStr} provided a solution for the ``real'' problem \textbf{\ref{th:prob_M2G}} as well as some other questions. First, they have proved the following.

\begin{theorem}
Multiset $A$ of $n$ numbers is uniquely determined by multiset of its $2$-sums $\sms A2$ if and only if $n$ is not a power of 2.
\end{theorem}

In other words, pair $(n, 2)$ is singular if and only if $n$ is a power of 2. Or, using our notation,
$$
\mcm_2 = \{4, 8, 16, 32, \ldots\} = \{2^k\,:\, k > 1\} \pp
$$

Second, they have explored case of $s=3$ of the more general problem \textbf{\ref{th:prob_MNK}}. They have shown that for $n=6$ there are easily constructed examples of different multisets of $n$ numbers $A$ and $B$ such that $A \eqs[0.2ex]3 B$. For instance:
\begin{align*}
A = \{1^5, -5\},\ B = \{(-1)^5, 5\} &\Longrightarrow A \neq B \\
\sms A3 = \{3^{10}, (-3)^{10}\},\ \sms B3 = \{(-3)^{10}, 3^{10}\} &\Longrightarrow \sms A3 = \sms B3 \pc
\end{align*}
where $1^5$ is not $1$ as you might have thought. In standard multiset notation $a^b$  means element $a$ with multiplicity $b$ (i.e., $b$-multiple entry of number $a$), so we simply mean that
$$
A = \{1,1,1,1,1,-5\}, \quad B = \{-1,-1,-1,-1,-1, 5\}\pp
$$

As for other values of $n$, the authors of \cite{SelStr} have proved that such example for case $s=3$ can exist only if $n^2 - (2^k+1)n+2\cdot 3^{k-1}$ vanishes for some natural $k<n$. It is not very hard to prove that the only other nontrivial values of $n$ for which that is possible are $n=27$ (with $k = 5,9$) and $n=486$ ($k=9$).

Thus, they have showed that
$$
\{6\}\subset \mcm_3 \subset \{6, 27, 486\}\pp
$$

Third, using the same technique for the case of $s=4$ it was shown that the only nontrivial values of $n$ when recovery might not be always possible are $n=8$,~$12$.

Presenting an example for $n=8$ is quite easy. Generally, one can always construct an example of ``recovery failure'' in \textbf{Problem \ref{th:prob_MNK}} if $n=2s$ (we will do that later in \textbf{Section \ref{sec:Simple_Examples}}). Again, this can be written as
$$
\{8\}\subset \mcm_4 \subset \{8, 12\}\pp
$$

Naturally, that suggested a few additional questions.

\begin{question}
\label{th:prob_s3}
Do pairs $(n,s) = (27,3)$ and $(486,3)$ represent actual ``recovery failures''? In other words, do there exist for $n=27$ and $n=486$ examples of different $n$-multisets $A$ and $B$ such that $\sms A3 = \sms B3$? 
\end{question}

\begin{question}
\label{th:prob_s4}
Same question about pair $(n,s) = (12,4)$. That is, do there exist two different $12$-multisets $A$ and $B$ such that $\sms A4 = \sms B4$? 
\end{question}

At that time both questions were left unsolved.

\smallskip

Yet another important question from the same article:
\begin{question}
\label{th:prob_mt2}
In cases when recovery is impossible, could there exist more than two $n$-multisets that generate identical multisets of $s$-sums?
\end{question}

Authors hypothesized that the answer to this one was negative.

\subsection*{(1959)}

Soon after the paper by Selfridge and Straus, Leo Moser and his coauthor Joachim (Jim) Lambek wrote a small article \cite{LamMos2}. It started by acknowledging results of their colleagues from UCLA, and then they proceeded to develop the problem in a slightly different direction.

Namely, they asked a question whether the set of non-ne\-ga\-tive integers can be split in two subsets $A = \{a_1, a_2, \ldots\}$ and $B = \{b_1, b_2, \ldots\}$ such that $\sms A2$ and $\sms B2$ coincide as multisets. They proved that the answer was positive and that there exists only one such decomposition of $\mbz_{\gte0}$. 

They did that by using multiset generating functions.

\begin{definition}
For any finite multiset $A$ of non-ne\-ga\-tive integers of the form $\{a_1^{k_1}, a_2^{k_2}, \ldots, a_m^{k_m}\}$ we define its generating function (polynomial) $f_A(x)$ by formula
$$
f_A(x) = \sum_{i=1}^m k_i x^{a_i} \pp
$$
Similarly this generating function can be defined for an infinite multiset $A = \{a_1^{k_1}, a_2^{k_2}, \ldots\}$ as long as sequence $\{k_i^{1/a_i}\}$ is bounded.
\end{definition}

Authors proved that in their particular case generating functions satisfied system of equations:
$$
\begin{cases}
f_A(x) + f_B(x)     &= 1/(1-x) \\
f_A^2(x) - f_A(x^2) &= f_B^2(x) - f_B(x^2) \pp
\end{cases}
$$

It was also proved that similar split of $Z_n = \{0,1,2,\ldots, n-1\}$ is possible if and only if $n$ is a power of 2. That split is unique and is determined by the so-called \textit{Thue-Morse sequence} $\{\alpha_n\}$ defined as $\alpha_n = s_2(n) \mmod2$ where $s_2(n)$ is the \textit{binary weight} of $n$, i.e., sum of digits (or simply, the number of ones) in the binary representation of $n$. So if $n = 2^p$ and we define sets $A = \{a_1, \ldots, a_m\}$ and $B = \{b_1, \ldots, b_m\}$ as follows
\begin{align*}
A &= \{k \in Z_n : \alpha_k = 0\} \\
B &= \{k \in Z_n : \alpha_k = 1\} 
\end{align*}
then $\sms A2 = \sms B2$. Indeed, if you set $f_A(x) = \sum_{i=1}^m x^{a_i}$ and $f_B(x) = \sum_{i=1}^m x^{b_i}$, then we have
$$
f_A(x) + f_B(x) = u(x) = \frac{1-x^{2^p}}{1-x}, \quad
f_A(x) - f_B(x) = v(x) = \prod_{i=0}^{p-1}(1-x^{2^i}) \pp
$$

Thus $f_A = (u+v)/2$ and $f_B = (u-v)/2$. From that it follows quite easily that $f_A^2(x) - f_A(x^2) = f_B^2(x) - f_B(x^2)$. It is left to notice that the sides in the the last equality are generating functions for multisets $\sms A2$ and $\sms B2$ respectively.

A somewhat similar ``generating functions'' approach was used later in \cite{GorFraStr}, \cite{BomBolOne} and \cite{FomIzh} in conjunction with some other ideas.

\subsection*{(1962)}

The next paper on the subject appeared in 1962, when Gordon, Fraenkel, and Straus published \cite{GorFraStr} proving that answer to \textbf{Question \ref{th:prob_mt2}} was positive. This was not the last time when Multiset Recovery Problem defied the expectations.

Authors have found numerous multi-sin\-gu\-la\-rity examples for the simplest case of $s=2$. More precisely, they have showed how to construct examples of three different 8-multisets $A$, $B$, and $C$ such that $\sms A2 = \sms B2 = \sms C2$. Here is one of these examples:
$$
A = \{0, 5, 6, 7, 9, 10, 11, 16\}; \ 
B = \{1, 4, 5, 6, 10, 11, 12, 15\}; \ 
C = \{2, 3, 4, 7, 9, 12, 13, 14\}.
$$

\smallskip

After this, naturally, the original question was adjusted into asking how many different $n$-multisets could generate the same multiset of $s$-sums. The maximum possible number of such multisets was denoted by $\mcf_s(n)$. Of course, the pair $(n,s)$ must be singular to begin with---which is equivalent to inequality $\mcf_s(n) > 1$.

Then more inequalities for $\mcf_s(n)$ were proved. For instance,
$$
\mcf_2(16) \lte 3,\quad 2 \lte \mcf_3(6) \lte 6,\quad \mcf_4(12) \lte 2\pp
$$
However, no other examples of triple (or greater) ``multiplicity'' were found, which led to another open question:

\begin{question}
\label{th:prob_mtx}
a) For $s=2$ does there exist $n = 2^p > 8$ such that some three distinct $n$-multisets generate the same multisets of $s$-sums (i.e., $\mcf_2(n) > 2$)?

b) Generally, does there exist any singular pair $(n,s)$ different from $(8,2)$ with three pairwise distinct $s$-equivalent $n$-multisets (i.e., $\mcf_s(n) > 2$)?
\end{question}

Two more results in the same article deserve mention. One was to prove that when dealing with any question about multiset recovery it was enough to work with the ring of integers $\mbz$, instead of arbitrary fields of characteristic zero or torsion-free Abelian groups. Another result resolved one of the questions about number of elements in $\mcm_s$. Namely, the authors proved that $\mcm_s$ was finite for all $s>2$.

\subsection*{(1962)}

The very same year the first part of the original \textbf{Problem \ref{th:prob_M2}} was used in one of the top math contests in the Soviet Union -- namely, Moscow City Mathematical Olympiad. It is quite possible that some Soviet mathematician had seen the  article \cite{GorFraStr} and liked the original question well enough to submit it to the olympiad committee. It was given to high school juniors and proved to be one of the more difficult problems of that year. An unpublished compilation of problems from that competition (translated into English) can be found in \cite{GalTol}.

\subsection*{(1968)}

Among two questions about suspect pairs (\textbf{Questions \ref{th:prob_s3}, \ref{th:prob_s4}}) the latter---$n = 12, s = 4$---seemed easier. So it was not surprising that ``only'' ten years after the original article \cite{SelStr}, John Ewell published his paper \cite{Ewe} claiming that pair $(12,4)$ was not singular and recovery was always possible for this case (it became a part of his Ph.D thesis). He also found a purely combinatorial and more direct proof of the important formula \eqref{eq:F_skn} (see below, in \textbf{Section \ref{sec:Sym_Poly}}).

Many years later further investigation uncovered an error in calculations regarding the pair $(12,4)$. However, another result in the same article was clearly correct---namely, Ewell demonstrated that the answer to \textbf{Question \ref{th:prob_mtx}(b)} was positive. He has proved that $\mcf_3(6) = 4$ and then went on to provide complete characterization of all possible quartets of pairwise different 3-equivalent 6-multisets.

We will give you one of these examples as a demonstration of Ewell's discovery
$$
A = \{0, 5, 9, 10, 11, 13\};
B = \{1, 5, 8, 9, 10, 15\};
C = \{1, 6, 7, 8, 11, 15\};
D = \{3, 5, 6, 7, 11, 16\} \pc
$$
leaving the actual verification as an easy exercise for the reader.

\subsection*{(1981)}

Richard Guy mentioned multiset recovery in his compendium of unsolved problems in number theory, see \cite{Guy}, \textbf{Problem C5}. He explained that it had been solved for $s=2$ while the question for values of $\mcf_2$ still had not been answered in full. Case $s=3$ for $n=27$ and $n=486$ was once again posed as an open question.

\subsection*{(1991)}

As far as we know, after Ewell's thesis Multiset Recovery Problem slipped into relative obscurity until 1991, when Boman, Bolker and O'Neil have explored a slightly different approach in their article \cite{BomBolOne}. 

More precisely, for point $x=(x_1, x_2, \ldots, x_n)$ in $n$-dimensional euclidean space $\mbr^n$ and for any $s$-subset $A = \{a_1,\ldots,a_s\}$ of $I_n = \{1,2,\ldots, n\}$ let us define $x_A$ as the sum $x_{a_1}+\ldots+x_{a_s}$. Then we can define linear operator 
$$
R_{n,s}:\mbr^n \to \mbr^{\binom ns}, \quad
R_{n,s}(x_1, x_2, \ldots, x_n) = (x_{A_1}, x_{A_2}, \ldots, x_{A_{\binom ns}})\pc
$$
where $A_1$, \dots, $A_{\binom ns}$ is the sequence of all $s$-subsets in $I_n$. This is a ``sort'' of discrete (combinatorial) version of Radon integral transform. Mapping $R_{n,s}$ can obviously be transferred from euclidean spaces to their reductions modulo standard actions (permutations of coordinates) of symmetric groups $S_n$ and $S_{\binom ns}$ respectively so that we have $R_{n,s}: \mathbb L_n \to \mathbb L_{\binom ns}$ where $\mathbb L_k = \mbr^k/S_k$. Then Multiset Recovery Problem can be posed as a question on whether $R_{n,s}$ is an injection. Or more generally, as a question on the size of $R_{n,s}^{-1}(x)$, with $x\in\mathbb L_{\binom ns}$.

Using this notation and terminology they proved---among other things---that when recovering $n$-multiset from the collection of its 2-sums one cannot obtain more than $n-2$ different multisets. Improving on that result they have also showed that for any $n \neq 8$ this upper bound could actually be lowered to 2 and almost always to 1. Thus they solved \textbf{Question \ref{th:prob_mtx}(a)}.

It is worth noting that judging by the list of the open questions, at that time the authors did not know about Ewell's paper \cite{Ewe}.

\subsection*{(1992)}

By some happy ``accident'' in 1991 \textbf{Question \ref{th:prob_M2G}} was used in a student mathematical contest in St.Petersburg, USSR. The author of this survey was---as surely many other mathematicians before him---fascinated by this seemingly simple problem, and started his own investigation. That resulted in article \cite{FomIzh} by Fomin and Izhboldin submitted for Russian publication in 1992 (English translation was published in 1995).

Most of that article was about rediscovering the very same results already achieved in \cite{SelStr} and \cite{GorFraStr}---unfortunately, due to a rather poor access to international scientific magazines the authors could not properly search for the papers already written on this issue. However, their article still contained one completely new result: singularity examples which positively answered \textbf{Question \ref{th:prob_s3}} for both cases ``under suspicion''. Pairs $(27,3)$ and $(486,3)$ were proved singular.

Thus, investigation of generalized Multiset Recovery Problem \textbf{\ref{th:prob_MNK}} for the case of $s=3$ was closed.

\subsection*{(1996)}

Just a few short years later, Boman and Linusson have independently come up with singularity examples for pairs $(27,3)$ and $(486,3)$ in \cite{BomLin}. Alas, they thought that case $(12,4)$ was already resolved by Ewell---at the end of their paper they mentioned that they were told (apparently at the very last moment) about article \cite{Ewe}. They also made some inroads into finding all possible singularity examples for $s=3$.

\subsection*{(1997)}

Ross Honsberger dedicated a chapter called ``A Gem from Combinatorics'' of his book \cite{Hon}  to the case $s=2$ of Multiset Recovery Problem. It is curious that he never mentions Leo Moser. Instead Honsberger stated that the results he had reproduced came from Paul Erdős and John Selfridge. This is the only time when Erdős's name appears in this story. It is not clear whether he really has done something there or possibly it was just a mistake in attribution.

\smallskip

The same year Tewodros Amdeberhan and Melkamu Zeleke (both from Temple University, Philadelphia, PA) have shown in \cite{AmZel} that the combinatorial Radon transform of order $5$ is almost always injective by proving that polynomials $\fsk 5k (n)$ have no other positive integer roots but $n = 2$, $3$, $4$, $5$, and~$10$. This closed the case of~$s=5$.

\subsection*{(2003)}

In chapter 46 of their engaging book \cite{SavAnd}, Savchev and Andreescu explained the solution for \textbf{Question \ref{th:prob_M2G}} and also went over the results from Lambek and Moser's article \cite{LamMos2} concerning Thue-Morse sequence. 

\subsection*{(2008)}

A slightly expanded version of \textbf{Question \ref{th:prob_M2N}} with extra items repeating parts of \cite{GorFraStr} and \cite{BomBolOne} was published as another problem in American Mathematical Monthly---submitted by Chen and Lagarias, \cite{CheLag}. Some of the solutions were subsequently posted and discussed on the Cut-The-Knot website, see \cite{Bog}. 

\subsection*{(2016)}

Nothing significant happened for quite some time until Isomurodov and Kokhas (\cite{IsoKok}) discovered that Ewell made a mistake in his lengthy polynomial computations for the pair~$(12,4)$. We will never know how that happened, but nowadays mathematicians no longer have to do all these exhausting calculations by hand---for instance, authors of \cite{IsoKok} made use of symbolic computational package \textsc{Maple}\texttrademark. After that the authors proved the existence of the ``recovery failure'' example and actually produced it, thus solving \textbf{Question \ref{th:prob_s4}} and finalizing case $s=4$ of \textbf{Problem \ref{th:prob_MNK}} (see below in \textbf{Section \ref{sec:Sym_Poly}}).

\section{Some simple examples}
\label{sec:Simple_Examples}

This short section explains how to construct some simple examples of singular pairs.

\subsection*{Case \texorpdfstring{$s=2$, $n = 2^k$}{s=2, n=2\textasciicircum k}} The most obvious and trivial of all examples of ``recovery failure'' is the pair $(2,2)$: one cannot hope to restore a set of two numbers knowing only their sum. This is not a ``real'' singular pair (because $n=s$) but we can use it as a basis from which less trivial examples are built.

Namely, if we have two $n$-multisets $A$ and $B$ which are 2-equivalent, then for any number $d$ we have
\begin{equation}
\label{eq:abd_eqs_adb}
A \cup (B+d) \eqs[0.2ex]2 B \cup (A+d)
\end{equation}
where $X+d$ is multiset obtained from $X$ by adding $d$ to all of its elements. So if we start with
$$
A=\{1,1\},\ B=\{0,2\},\ A \eqs[0.2ex]2 B
$$
then choosing $d=1$ we get
$$
A'=\{1,1,1,3\},\ B'=\{0,2,2,2\},\ A' \eqs[0.2ex]2 B' \pp
$$
Proceeding in this manner, we can easily build examples of  2-equivalent $n$-multisets for any $n$ which is a power of 2. Again, we will leave the proof of \eqref{eq:abd_eqs_adb} as an exercise for the reader.

\subsection*{Case \texorpdfstring{$n = 2s$}{n=2s}} Remember that singularity example for $n=6$, $s=3$ from \textbf{Section \ref{sec:Timeline}}? It can be easily generalized for any pair $(n,s)$ where $n=2s$.

Namely, you can take some $2s$-multiset $A$, find its arithmetic mean $a$ and reflect $A$ with respect to $a$ to obtain what we will call its \textit{mirror} multiset $\tilde A = 2a - A$. As long as $A$ is not symmetric, $\tilde{A}$ will be different from $A$ and $A \eqs s \tilde{A}$. To prove that it is sufficient to notice that for each $M\subset A$ with $|M|=s$ the mirror image of $A\backslash M$ (which is a sub-multiset of $\tilde{A}$ consisting of $s$ numbers) has the same sum of elements.

For demonstration purposes we only need one example---let us consider $A = \{1^{2s-1}\pc 1-2s\}$ and its mirror $\tilde{A} = \{(-1)^{2s-1}\pc 2s-1\}$.

All the $s$-sums of numbers in $A$, there are $\binom{2s}{s}$ of them, fall into two groups. One consists of the sums that include element $(1-2s)$---there are $\binom{2s-1}{s-1}$ of those, and each one of these sums is equal to $(s-1)\cdot 1 + (1-2s) = -s$. The other one has $\binom{2s-1}{s}$ sums that do not have $(1-2s)$ in them, each one of them equal to $s$. Thus we have $\sms As = \{(-s)^m, s^m\}$ where $m = \binom{2s-1}{s-1} = \binom{2s-1}{s}$. You can see that $\sms As$ is symmetric (with respect to zero) and thus $\sms{\tilde{A}}s = \sms As$.

\subsection*{Duality \texorpdfstring{$(n,s) \leftrightarrow (n, n-s)$}{(n,s)<-->(n,n-s)}} 

If we have two $s$-equivalent $n$-multisets $A$ and $B$, then these same multisets are $(n-s)$-equivalent as well. To prove that, it is sufficient to demonstrate that the sum of all elements of $A$ equals to that of $B$. Quick computation shows that
$$
\sum_{1\lte i_1 < i_2 \cdots < i_s \lte n} (a_{i_1}+a_{i_2}+\ldots+a_{i_s})
  = \binom{n-1}{s-1} \sum_{1\lte i \lte n} a_i \pp
$$

Thus the sum of numbers in $A$ equals the sum of numbers in $B$ and denoting that number by $S$ we have
$$
\sms A{n-s} = S - \sms As, \quad \sms B{n-s} = S - \sms Bs \pc
$$
which proves the duality. This means we can always assume that $n \gte 2s$; if $s < n < 2s$, then we can switch to the pair $(n, s')$ where $s' = n-s$ and $n > 2s'$.

This duality allows us to generate more examples of singular pairs. For instance, since $(8, 2)$ is singular then $(8, 6)$ is singular, too. As we will see soon, the pairs $(27, 3)$, $(486, 3)$, and $(12, 4)$ are singular---therefore, the pairs $(27, 24)$, $(486, 483)$, and $(12, 8)$ are singular as well.

Later in this article (see \textbf{Section \ref{sec:Roots_Fsk}}) we will talk more about this duality and its partial expansion.

\subsection*{Linear transformations} 

Finally, one obvious but useful fact. If $A\eqs s B$ and $f(x)=px+q$ is some arbitrary linear function, then the multisets $f(A)$ and $f(B)$ are also $s$-equivalent. That simply means we can translate and stretch/shrink singularity examples to obtain new ones. 

For instance, if you consider $A = \{0, 5, 9, 10, 11, 13\}$, $B = \{1, 5, 8, 9, 10, 15\}$, then  $A\eqs[0.2ex]3 B$. Applying $f(x)=2x-13$ we obtain a new pair of 3-equivalent multisets $A_1 = \{-13, -3, 5, 7, 9, 13\}$, $B_1 = \{-11, -3, 3, 5, 7, 17\}$.

Of course, all the singularity examples that can be obtained from each other by such operations will be considered identical for the purposes of this investigation.

\section{Moser polynomials}
\label{sec:Sym_Poly}

Now let us delve into specific techniques used in multiset recovery. The main one is based on the following approach that utilizes symmetric polynomials.

Given $n$-multiset $A = \{a_1, \ldots, a_n\}$ we can produce a sequence of sums of its $k$-th powers for $k=1,\ldots,n$. That is, we can apply power-sum symmetric polynomials in $n$ variables
$$
\sigma_k(x_1, \ldots, x_n) = \sum_{i=1}^n x_i^k
$$
to multiset $A$ to obtain sequence $\sigma_1(A), \ldots, \sigma_n(A)$. It is well known that $A$ is uniquely defined by this sequence---values of $\sigma_k(A)$ determine coefficients of polynomial $(x-a_1)(x-a_2)\ldots(x-a_n)$ and therefore determine the multiset of its roots.

Thus if all values of $\sigma_k(A)$ for $k = 1,\ldots, n$ can be deduced from values of $\sigma_k(\sms As)$, then multiset $A$ is uniquely determined by multiset $\sms As$.

Let us start from small values of $k$. For $k=1$ we have already computed
$$
\sigma_1(\sms As) = \binom{n-1}{s-1}\sigma_1(A)
$$
and therefore, if $\sms As = \sms Bs$, then $\sigma_1(A) = \sigma_1(B)$. This means that $\sigma_1(A)$ can always be found from $\sms As$. 

\medskip

Now, if $k=2$, then 
\begin{align*}
\sigma_2(\sms As) &= \sum_{1\lte i_1 < \cdots < i_s \lte n}(a_{i_1}+\ldots+a_{i_s})^2 =
\binom{n-1}{s-1}\sum_{1\lte i \lte n} a_i^2  + \binom{n-2}{s-2}\sum_{1\lte i<j\lte n}2a_ia_j = \\
&= \binom{n-1}{s-1}\sum_{1\lte i \lte n} a_i^2  + \binom{n-2}{s-2}\sum_{1\lte i\neq j\lte n}a_ia_j = \\
&= \left(\binom{n-1}{s-1}-\binom{n-2}{s-2}\right)\sum_{1\lte i \lte n} a_i^2  + \binom{n-2}{s-2}\sum_{1\lte i,j\lte n}a_ia_j = \\
&= \binom{n-2}{s-1}\sigma_2(A) + \binom{n-2}{s-2}\sigma_1(A)^2 \pp
\end{align*}

We already know $\sigma_1(A)$ and thus we can find $\sigma_2(A)$ as long as coefficient $\binom{n-2}{s-1}$ is not zero. For $n>s>0$ that is always true and therefore $\sigma_2(A)$ is also always ``recoverable''.

Since $\sigma_k(\sms As)$ is a symmetric polynomial in $a_1$, \dots $a_n$, it can be expressed in the following way:
\begin{equation}
\label{eq:sigma_As_poly}
\sigma_k(\sms As) = \alpha_{s,k,n} \sigma_k(A) +
 \mathcal P(\sigma_1(A), \sigma_2(A), \ldots, \sigma_{k-1}(A)) \pc
\end{equation}
where $\alpha_{s,k,n}$ is a constant (in terms of variables $a_i$) defined by three numbers $n, s, k$, and $\mathcal P$ is some polynomial in $k-1$ variables $\sigma_1$, \dots, $\sigma_{k-1}$, whose coefficients are fully defined by that triplet as well. For instance, as we have just shown, for $k=2$ we have $\alpha_{s,k,n} = \binom{n-2}{s-1}$ and $\mathcal P(\sigma) = \binom{n-2}{s-2}\sigma^2$.

It follows that if we could show that coefficient $\alpha_{s,k,n}$ does not vanish, then we will have proved that $\sigma_k(A)$ is determined by $\sigma_1(A)$, $\sigma_2(A)$, \dots, $\sigma_{k-1}(A)$, $\sigma_k(\sms As)$. Thus, if $\alpha_{s,k,n} \neq 0$ for all $1\lte k \lte n$, then using this for $k=1$, then for $k=2$ etc, we can conclude that multiset $A$ can be recovered from multiset $\sms As$ and the pair $(n,s)$ is not singular.

The following equality is true:

\begin{theorem}
\label{th:F_skn}
\begin{equation}
\label{eq:F_skn}
\alpha_{s,k,n} = \sum_{p=1}^s (-1)^{p-1}p^{k-1}\binom{n}{s-p}\pp
\end{equation}
\end{theorem}

Theorem \ref{th:F_skn} was proved in \cite{GorFraStr} by some neat manipulation of a formula from \cite{SelStr}. Later a purely combinatorial and more direct proof of \eqref{eq:F_skn} was given in \cite{Ewe}. And then it was again ``rediscovered'' and proved (in a somewhat different manner, by making use of exponential generating functions) in \cite{FomIzh}.

So the coefficient $\alpha_{s,k,n}$ turns out to be a polynomial in $n$.

\begin{definition}
The right side of equation \eqref{eq:F_skn} will be called the \textbf{Moser polynomial} and will be denoted by $\fsk sk(n)$.
$$
\fsk sk(n) = \sum_{p=1}^s (-1)^{p-1}p^{k-1}\binom{n}{s-p}\pp
$$
We will also set $\fsk sk(n) = 0$ for any integer $s < 1$. In this way Moser polynomials are defined for any integer number $s$ and natural number $k$.
\end{definition}

We will leave it to the reader as a simple exercise to prove that for $k=2$ formula \eqref{eq:F_skn} is indeed equivalent to $\fsk s2(n) = \binom{n-2}{s-1}$.

\medskip

Incidentally, even without this formula case $s=2$ (\textbf{Problem \ref{th:prob_M2G}}) can now be resolved in a very straightforward manner. Computing $\sigma_k(\sms A2)$ we obtain (using notation from \eqref{eq:sigma_As_poly})
$$
\alpha_{2,k,n} = n - 2^{k-1}\pc
$$
which means that $n$-multiset is always recoverable from the multiset of its 2-sums if $n$ is not a power of 2. As we already know, if $n$ is a power of 2, then $n$-multiset $A$ cannot always be recovered from $\sms A2$.

\subsection*{Case \texorpdfstring{$s=3$}{s=3}} Equation \eqref{eq:F_skn} gives us
\begin{align*}
\fsk 3k(n)  &= \binom{n}{2} - 2^{k-1}\binom{n}{1} + 3^{k-1} \pc \\
\vfs
2\fsk 3k(n) &= n^2 - n(2^k+1) + 2\cdot3^{k-1}\pp
\end{align*}

Investigation here is again relatively straightforward. First, we can prove that $\fsk 3k(n)$ cannot be zero for positive integer $n$ if $k > 12$. Second, we check all the cases with $k\lte 12$ and verify that polynomials $\fsk 3k$ have integer roots if and only if $k \in \{1, 2, 3, 5, 9\}$. For these five special cases we have 
\begin{align*}
\fsk 31(n) &= \frac12(n-1)(n-2)\pc \\
\fsk 32(n) &= \frac12(n-2)(n-3)\pc \\
\fsk 33(n) &= \frac12(n-3)(n-6)\pc \\
\fsk 35(n) &= \frac12(n-6)(n-27)\pc \\
\fsk 39(n) &= \frac12(n-27)(n-486)\pp
\end{align*}

From \textbf{Section \ref{sec:Simple_Examples}} we already know that the pair $(6,3)$ is singular. To prove the same for pairs $(27,3)$ and $(486,3)$ we present the following examples:
$$
A'_{27} = \{0, 1^{16}, 2^{10}\} \pc
\ A''_{27} = \{0^5, 1^{10}, 2^{10}, 3^2\} \pc
\ A'''_{27} = \{0, 1^5, 2^{10}, 3^6, 4^5\} \pc
\ A_{486} = \{0^{22}, 1^{176}, 2^{231}, 3^{56}, 4\} \pp
$$

We will leave it to the reader to verify that each one of these four multisets is 3-equivalent to its mirror. Nowadays, this can be done in minutes, using just a few lines of code in some decent computational package.

Summary: $\mcm_3 = \{6, 27, 486\}$.

\subsection*{Case \texorpdfstring{$s=4$}{s=4}} From \eqref{eq:F_skn} we obtain
\begin{align*}
\fsk 4k(n) &= \binom{n}{3} - 2^{k-1}\binom{n}{2} + 3^{k-1}\binom{n}{1} - 4^{k-1} \pc \\
\vfs
6\fsk 4k(n)&= n^3 - n^2(3\cdot2^{k-1}+3) + n(2\cdot3^k + 3\cdot2^{k-1} + 2) -6\cdot4^{k-1}\pp
\end{align*}

Using divisibility and other relatively straightforward number theory ideas we can prove that for $k>7$ polynomials $\fsk 4k$ do not have positive integer roots. And, finally,
\begin{align*}
\fsk 41(n) &=\frac16(n-1)(n-2)(n-3)\pc \\
\fsk 42(n) &=\frac16(n-2)(n-3)(n-4)\pc \\
\fsk 43(n) &=\frac16(n-3)(n-4)(n-8)\pc \\
\fsk 44(n) &=\frac16(n-4)(n^2-23n+96)\pc \\
\fsk 45(n) &=\frac16(n-8)(n^2-43n+192)\pc \\
\fsk 46(n) &=\frac16(n-12)(n^2-87n+512)\pc \\
\fsk 47(n) &=\frac16(n-8)(n^2-187n+3072)\pp
\end{align*}

Case $n=8$ is ``trivial''---this is the situation $n = 2s$ which is well known to us by now. The only other nontrivial root of $\fsk 4k$ polynomials is 12. As we mentioned before, this case turned out to be tougher than the others---a pair of 4-equivalent 12-multisets was found only in 2016 by Isomurodov and Kokhas. Namely, if we consider the two following different 12-multisets
$$
A = \{1^2, 4, 6, 7, 8^2, 9, 10, 12, 15^2\},\quad
B = \{0, 3, 4, 5, 6, 7, 9, 10, 11, 12, 13, 16\} \pc
$$
then $A \eqs[0.2ex]4 B$.

In \cite{IsoKok} the authors have actually proved that this is the only possible singularity example for case $(12, 4)$ (considering pairs of multisets that differ only by linear transformation to be identical).

Summary: $\mcm_4 = \{8, 12\}$.

\subsection*{Case \texorpdfstring{$s=5$}{s=5}} In this case we have
\begin{align*}
\fsk 5k(n) &= \binom{n}{4} - 2^{k-1}\binom{n}{3} + 3^{k-1}\binom{n}{2} - 4^{k-1}\binom{n}{1} + 5^{k-1} \pc \\
\vfs
24\fsk 5k(n)&= n^4 - n^3(2^{k+1}+6) + n^2(4\cdot 3^k + 3\cdot2^{k+1} + 11) -  n(6\cdot 4^k + 4\cdot 3^k + 4\cdot 2^k + 6) + 24\cdot 5^{k-1} \pp
\end{align*}

See  \cite{AmZel} for the proof that with the exception of $n=5$ and $n=10$ polynomials $\fsk 5k (n)$ have no other integer roots greater than or equal to~$5$.

Summary: $\mcm_5 = \{10\}$.

\section{Digging for roots (of Moser polynomials)}
\label{sec:Roots_Fsk}

Since we know that pair $(n,s)$ can be singular only if $n$ is a root of $\fsk sk$ for some $1\lte k \lte n$, then let us turn our attention to finding out more about those roots. (The rest of this section was inspired by the proof of \textbf{Theorem 7} from \cite{SelStr}).

For this let us take another, closer, look at the Moser polynomials for the first few values of $k$.

\subsection*{Case \texorpdfstring{$k=1$}{k=1}} We already know that 
$$
\fsk s1(n) = \binom{n-1}{s-1} = \frac1{(s-1)!}(n-1)(n-2)\ldots(n-s+1) = \frac1{(s-1)!}  \prod_{p=1}^{s-1} (n-p) \pp
$$

Thus the roots are 1 through $s-1$ and of no interest to us---$n$ has to be greater than $s$ to provide us with a possibly singular pair $(n,s)$.

\subsection*{Case \texorpdfstring{$k=2$}{k=2}} We have also computed this one before.
$$
\fsk s2(n) = \binom{n-2}{s-1} = \frac1{(s-1)!} \prod_{p=2}^{s} (n-p) \pp
$$

Again, no roots of interest. Let us go on.

\subsection*{Case \texorpdfstring{$k=3$}{k=3}} It is still fairly easy to compute
$$
\fsk s3(n) = \frac1{(s-1)!}(n-2s) \prod_{p=3}^{s} (n-p) \pc
$$
which (finally!) has a nontrivial root $n=2s$. However, we already know about it---the pair $(2s, s)$ is always singular.

\subsection*{Case \texorpdfstring{$k=4$}{k=4}} This computation might take a little longer but eventually you will get the following formula (for $s > 2$)
\begin{equation}
\label{eq:k4_fsk}
\fsk s4(n) = \frac1{(s-1)!}(n^2-(6s-1)n+6s^2) \prod_{p=4}^{s} (n-p) \pp
\end{equation}

A-ha! We now have a quadratic Diophantine equation for nontrivial roots $n$
\begin{equation}
\label{eq:k4_nsq}
n^2-(6s-1)n+6s^2 = 0 \pc
\end{equation}
which can be also rewritten as a quadratic equation for $s$
\begin{equation}
\label{eq:k4_ssq}
s^2 - sn + \frac{n(n+1)}6 = 0 \pp
\end{equation}

The sum of the roots of equation \eqref{eq:k4_ssq} is $n$---hence, if the pair $(n, s)$ with $n>s$ is a root of equation \eqref{eq:k4_fsk}, then so is the pair $(n, n-s)$. This is clearly a direct analog of the singular pairs' duality we have described in \textbf{Section \ref{sec:Simple_Examples}}. Let us call such pairs \textit{$\langle n,4\rangle$-\textbf{conjugated}} or simply $n$-\textbf{conjugated}.

In the same manner from equation \eqref{eq:k4_nsq} we can conclude that if the pair $(n, s)$ is a root of equation \eqref{eq:k4_fsk}, then its other root is the pair $(6s-1-n, s)$. These two pairs will be called \textit{$\langle s,4\rangle$-\textbf{conjugated}}.

Obviously, both types of conjugation are symmetric. It is also clear from equations \eqref{eq:k4_nsq} and \eqref{eq:k4_ssq} that for any positive solution $(n,s)$ we have $6s-1>n$ and $n>s$---otherwise the left sides of these equations are positive. Thus, conjugate pair always consists of two positive integers as well.

Let us consider the smallest possible root of \eqref{eq:k4_nsq}, namely $r = (2,1)$ (since we are solving the equation in positive integers, we are allowed to talk about ``smallest'' solution). That pair is not something we can directly use because $s$ must be at least 3 for the formula \eqref{eq:k4_fsk} to make sense. But it is still a root of our quadratic equation \eqref{eq:k4_nsq} and we will use it to produce others.

It is important to mention that the pair $r$ is self-$n$-conjugated ($2-1=1$) so the only way to produce a different solution is via $s$-conjugation. So we jump to pair $(3,1)$, then through $n$-conjugation to pair $(3, 2)$, then to $(8, 2)$, then to $(8, 6)$, and so on. 

Proceeding like that we will obtain one infinite chain of solutions of equation \eqref{eq:k4_nsq}:
\begin{multline}
\label{chain4}
\underline{(2,1)} \conj{s} \underline{(3, 1)} \conj{n} \underline{(3, 2)} \conj{s} \underline{(8,2)} \conj{n} (8, 6) \conj{s} (27, 6) \conj{n}
      \\
(27, 21) \conj{s} (98, 21) \conj{n} (98, 77) \conj{s} (363, 77) \conj{n} (363, 286) \conj{s} \cdots 
\end{multline}
We have marked the arrows with small letters "$n$" and "$s$" to show which type of conjugation was used in each case; also we have underlined pairs which are not ``fully compliant''---they are roots of equation \eqref{eq:k4_nsq} but they are not roots of the corresponding polynomial $\fsk sk$ with $s>2$ (the pair $(8,2)$ is singular but we have discounted it because of the  requirement that $s>2$). So, starting from $(8, 6)$, pairs in the chain represent valid roots of the polynomials $\fsk s4$. Thus, they are all ``suspect'' as possible singularities for Multiset Recovery Problem.

\subsection*{Case \texorpdfstring{$k=5$}{k=5}} In this case computation is also not terribly complicated. For $s>3$ we obtain
\begin{equation*}
\label{eq:k5_fsk}
\fsk s5(n) = \frac1{(s-1)!}(n^2-(12s-5)n+12s^2)(n-2s) \prod_{p=5}^{s} (n-p) \pp
\end{equation*}

Again we have a quadratic Diophantine equation which can be written like this
\begin{equation}
\label{eq:k5_nsq}
n^2-(12s-5)n+12s^2 = 0 \pp
\end{equation}
or like this
\begin{equation}
\label{eq:k5_ssq}
s^2 - sn + \frac{n(n+5)}{12} = 0 \pp
\end{equation}

As before, equation \eqref{eq:k5_ssq} gives us $n$-conjugation ``duality'' $(n, s) \leftrightarrow (n, n-s)$. And equation \eqref{eq:k5_nsq} provides us with $\langle s,5\rangle$-conjugation---namely, $(n, s) \leftrightarrow (12s-5-n, s)$.

Similarly to the previous subsection we obtain infinite chain of solutions:
$$
\underline{(3,1)} \conj{s} \underline{(4, 1)} \conj{n} \underline{(4, 3)} \conj{s} \underline{(27,3)} \conj{n} (27, 24) \conj{s} (256, 24) \conj{n} (256, 232) \conj{s} (2523, 232) \conj{n} \cdots
$$

However this case differs somewhat from the previous one. The ``minimum solution'' pair $(3,1)$ has $n$-conjugate $(3,2)$ that does not coincide with it---thus the chain can be extended in the other direction as well. Therefore we obtain more solutions:
$$
\underline{(3,2)} \conj{s} \underline{(16, 2)} \conj{n} (16, 14) \conj{s} (147,14) \conj{n} (147, 133) \conj{s} (1444, 133) \conj{n} (1444, 1311) \conj{s}  \cdots
$$
and so our two chains can be merged into one
\begin{multline}
\label{chain5}
\cdots  
(1444, 1311) \conj{s} (1444, 133) \conj{n} (147, 133) \conj{s} (147,14) \conj{n} (16, 14) \conj{s}
\\
\conj{s} \underline{(16, 2)} \conj{n} \underline{(3,2)} \conj{s} \underline{(3,1)} \conj{s} \underline{(4, 1)}  \conj{n}  \underline{(4, 3)} \conj{s} \underline{(27,3)} \conj{n} 
\\
\conj{n} (27, 24) \conj{s} (256, 24) \conj{n} (256, 232) \conj{s} (2523, 232) \cdots
\end{multline}

\smallskip

In both cases $s=4$ and $s=5$ it is easy to prove that all positive integer solutions of equations \eqref{eq:k4_nsq} and \eqref{eq:k5_nsq} belong to the chains \ref{chain4} and \ref{chain5} respectively. We will leave that to the reader.

\subsection*{Case \texorpdfstring{$k=6$}{k=6}} It would be great if the same ideas could be applied for this and subsequent cases as well. However, the computation of $\fsk s6$ shows that for $s>4$ we have the following formula:
\begin{equation*}
\label{eq:k6_fsk}
\fsk s6(n) = \frac1{(s-1)!} g_{s,6}(n) \prod_{p=6}^{s} (n-p) \pc
\end{equation*}
where
$$
g_{s,6}(n) = n^4-(30s-16)n^3+(150s^2-90s+11)n^2-(240s^3-90s^2+4)n+120s^4 \pp
$$

Some of the polynomials $g_{s,6}$ have integer roots. For instance,
\begin{align*}
g_{8,6}(n) &= (n-12)(n^3-212n^2+6347n-40960) \pc \\
g_{10,6}(n) &= (n-32)(n^3-252n^2+6047n-37500) \pc \\
g_{22,6}(n) &= (n-32)(n^3-612n^2+51047n-878460) \pc \\
g_{30,6}(n) &= (n-32)(n^3-852n^2+105047n-3037500) \pp
\end{align*}

This is basically all we get for $k=6$. Alas, no more quadratic Diophantine equations, no chains of conjugation. Also, it seems likely that polynomials $g_{s,6}$ do not have integer roots other than the ones shown above.

And, of course, the same happens with cases of even greater values of $k$---and so this line of investigation ends here.

To conclude this section, here are several useful facts about Moser polynomials. For the sake of brevity we will omit the proofs, leaving them to the reader.

\begin{prop}
Prove (preferably without using formula \eqref{eq:F_skn}) the following recurrence equations
\begin{align*}
\fsk sk(x) &= \fsk sk(x-1) + \fsk {s-1}k(x-1) \\
\fsk sk(x) &= s \fsk s{k-1}(x) - x \fsk {s-1}{k-1}(x-1) \pp
\end{align*}
\end{prop}

\begin{prop}
Prove that for $k>1$ the polynomial $\fsk sk(x)$ is divisible by $\prod_{p=k}^{s} (x-p)$.
\end{prop}

\begin{prop}
Prove that if $n\gte k > 1$, then $\fsk sk(n) = (-1)^k\fsk {n-s}k(n)$.
\end{prop}

\section{Computer to the rescue}
\label{sec:Computer_Results}

\subsection*{Roots of \texorpdfstring{$\fsk sk$}{F\_ sk}}

Trying to find more roots of Moser polynomials for cases $s > 4$ in hope of some insight, I have written a short program in \SAGE{} which was then run through SageMath web interface at \href{http://cocalc.com}{CoCalc.com} for $s=3,4,5,6,7$, etc. until the server started to stumble (which happened somewhere around $s=40$). After that I have switched to a local install of \SAGE{} and proceeded until $s=\LastS$ when every new value of $s$ started to require almost a day to process (and then my computer ran out of operational memory).

The program did the following: for every fixed value of $s$ it ran the loop for $k$ from 1 to 1000, where at each step it computed polynomial $\fsk sk$, factored it over $\mbz$ and in the case of nontrivial factorization printed out the roots of the polynomial. At the end it also produced $k_{\max}(s)$---the last value of $k$ for which a nontrivial factorization of $\fsk sk$ occurred.

Below (in Table \ref{table_nsk}) you can see the summary of all nontrivial roots (with pairs $(2s,s)$ excluded) obtained from this experiment.

{
\newcommand{\km}[1]{${}^{[#1]}$}
\begin{table}[ht]
\begin{tabular}{| l | l |}
\hline \phn
  $s$ & \quad $n$ \\
\hline \phn
  3 & 27\km5, 486\km9 \\
\hline \phn
  4 & 12\km6 \\
\hline \phn
  6 & 8\km4, \textbf{27}\km4 \\
\hline \phn
  8 & 12\km6\\
\hline \phn
  10 & \textbf{32}\km6 \\
\hline \phn
  14 & 16\km5, \textbf{147}\km5 \\
\hline \phn
  21 & \textbf{27}\km4, \textbf{98}\km4 \\
\hline \phn
  22 & \textbf{32}\km6\\
\hline \phn
  24 & 27\km5, \textbf{256}\km5 \\
\hline \phn
  30 & 32\km6 \\
\hline \phn
  62 & 64\km7 \\
\hline \phn
  77 & \textbf{98}\km4, \textbf{363}\km4 \\
\hline \phn
  126 & 128\km8 \\
\hline \phn
  133 & \textbf{147}\km5, \textbf{ 1444}\km5 \\
\hline
\end{tabular}
\medskip
\caption{Nontrivial roots of $\fsk sk$ for $3 \lte s \lte \LastS$}
\label{table_nsk}
\end{table}
}

We have marked each entry $n$ with the first value of $k$ for which $\fsk sk(n)=0$. So, for instance, mark ${}^{[4]}$ corresponds to chain \eqref{chain4}, and ${}^{[5]}$---to chain \eqref{chain5}.

For all other values of $s$ between 3 and $\LastS$ the only roots found were either $n=2s$ (which would have been marked with ${}^{[3]}$) or trivial (1 through $s$) and therefore of no interest for us.

Roots of $\fsk sk$ which have not been verified yet as multiset recovery singularities are emphasized in \textbf{bold}. They represent the current ``suspect'' cases.

In addition, the experiment showed that for all $3 < s \lte \LastS$ the value of $k_{\max}(s)$ was equal to $2s-1$. Claiming that to be always true is what we will call the \textbf{$k_{\max}$-Conjecture}---see \textbf{Conjecture \ref{th:prob_kmax}} below, in \textbf{Section \ref{sec:Open_Questions}}.

The following proposition can be considered as a very easy ``half'' of this conjecture.

\begin{prop}
\label{th:oddk_Fsk2s}
For any $s>2$, $n=2s$, and any odd $k$ such that $1 < k < n$ we have $\fsk sk(n) = 0$. 
\end{prop}

\begin{proof}
We can rewrite this statement by using \eqref{eq:F_skn}, adding summand with $p=0$, and reversing the summation index $p$. As a result we obtain
\begin{equation*}
\sum_{p=0}^s (-1)^{p}(s-p)^r\binom{n}{p} = 0
\end{equation*}
for any even number $0 < r < n$.

Now, since 
$$
\binom{n}{p} = \binom{n}{n-p}, \quad
(s-p)^r = (s-(n-p))^r \pc 
$$
the equation above is equivalent to
\begin{equation*}
\sum_{p=0}^n (-1)^{p}(s-p)^r\binom{n}{p} = 0 \pp
\end{equation*}

Any polynomial in $p$ of degree less than $n$ (such as $(s-p)^r$) can be expressed as a linear combination of ``falling powers'' polynomials $p^{[j]}$, $j=0,\ldots, n-1$, where $p^{[j]} = p\cdot(p-1)\cdot \ldots \cdot(p-j+1)$ (another commonly used notation for that expression is $(p)_j$). Our proposition then follows from well-known formula
$$
\sum_{p=0}^n (-1)^{p}p^{[j]}\binom{n}{p} = 0 \pc
$$
which can be easily proved using the generating function $\lambda(x)=(1+x)^n=\sum_{p=0}^n \binom{n}{p}x^p$. Polynomial $\lambda(x)$ has $-1$ as a root of order $n$; thus its $j$-th derivative $\lambda^{(j)}(x)=\sum_{p=0}^n \binom{n}{p}p^{[j]}x^{p-j}$ also must have $-1$ as a root for any $0 \lte j < n$.

\end{proof}

Finally, from $\fsk s{2s-1}(2s) = 0$ follows
\begin{corollary}
For any $s>2$ we have $k_{\max}(s) \gte 2s-1$.
\end{corollary}

\subsection*{Singularity search}

Well, since we already started using computer assistance, let us continue down this slippery slope. The next idea in automating our investigation is to hunt not for the roots of polynomials $\fsk sk$ but for the singular multisets themselves.

The objective is to try and find singularity examples for the smallest ``suspect'' pairs $(27, 6)$ and~$(32, 10)$. The other suspects, not $n$-conjugated to these two, are too large to hope for any ``brute force'' computer search to succeed.

The main idea of this approach is to restrict the realm of the $n$-multisets that we deal with. 
Consider all $\binom{n+m-1}{m-1}$ \textit{weak compositions} of $n$ into $m$ parts, that is, representations of $n$ as a sum of $m$ non-ne\-ga\-tive integers $k_i$, $(i = 1, 2, \ldots, m)$:
\begin{equation*}
\label{eq:wcs_n_m}
\mcp:\ n = k_1+k_2+\ldots+k_m,\quad k_i\in \mbz_{\gte0} \pp
\end{equation*}

Each weak composition of this form can be treated as sequence of multiplicities---that is, from each composition $\mcp$ we will construct $n$-multiset
\begin{equation}
\label{eq:msets_n_m}
A_{\mcp} = \{1^{k_1},2^{k_2},\ldots,m^{k_m}\}\pp
\end{equation}

Alas, in both cases $(27, 6)$ and $(32, 10)$ we cannot hope to find $n$-multiset $X$ which is $s$-equivalent to its own mirror $\tilde{X}$. Indeed, if $X \eqs s \tilde{X}$, then without loss of generality we can assume that $\sigma_1(X)=0$. Thus $\tilde{X} = -X$ and for any even $k > 0$ the sum of $k$-th powers of numbers in $X$ will be equal to that of $\tilde{X}$.

From the experiment we already know (and can easily verify this formally) that $\fsk 6k(27) = 0$ if and only if $k=4$, and $\fsk {10}k(32) = 0$ if and only if $k=6$. Therefore, for any $X$ such that $\sms Xs = \sms{\tilde{X}}s$ we will have $\sigma_k(X) = \sigma_k(\tilde{X})$ for all values of $1\lte k \lte n$---the only exceptions we could have hoped for were 4 and 6 (for $n=27$ and $n=32$ respectively) and we have just eliminated them. 

\smallskip

So, how can we proceed and what are the challenges?

First of all, we cannot afford to generate an array of all weak compositions (or multisets of type \eqref{eq:msets_n_m}) and then analyze the result---the computer would soon run out of memory.  For example, if $n=27$ and $m=10$, then we get almost 100 million of such compositions (94,143,280 to be exact).

Thus, the algorithms here have to be iterative. It is fairly easy to write an iterator function which generates next weak composition based on the previous one. \textit{Hint:} find the last nonzero part, increase the previous part by one and make all the following parts zero except for the last one. If necessary, the same can be done when going through all $s$-subsets in an $n$-multiset.

Second, the check function that verifies whether the two given multisets $A$ and $B$ are $s$-equivalent must be written very carefully and very efficiently because it will be called quite a few times. To make it work as fast as possible the function needs to implement some ``quick rejection'' checks. For instance, the sum of the first $s$ numbers from $A$ (let us assume it is sorted) is always equal to the minimum number in $\sms As$; hence these sums for $A$ and $B$ must coincide. The more simple checks of this sort are employed the better.

Third, calling this check function for every pair $A_{\mcp}$ and $A_{\mcp'}$ of constructed multisets is absolutely out of the question. So, some sort of simplified ``signature'' has to be computed for each multiset $\sms{A_{\mcp}}s$ (alas, no quick rejections there) so we can compare these numbers instead of comparing very large multisets $\sms{A_{\mcp}}s$. But even with that we cannot go much farther beyond $m=10$ for the reason I already mentioned above---such huge arrays of data will exhaust the computer memory.

My own implementation of this approach did not find any examples of 6-equivalent 27-multisets of type \eqref{eq:msets_n_m} for $m < 10$. As a sanity check I ran the same code for the pair $(12,4)$ with $m=17$, and after a few hours of number crunching it resulted in the same unique example of two 4-equivalent 12-multisets already found in \cite{IsoKok}.

Clearly, absence of positive results in this computational experiment doesn't mean much as  there could be 6-equivalent 27-multisets that span longer stretches of integers. And it is always possible that singularity example for $(27,6)$ simply doesn't exist. For now, this remains an open question. 

If some of the readers become interested in this line of investigation I will gladly send them my code---I am sure it can be made more effective while consuming less memory. And then, who knows, perhaps the next value of $m$ will finally yield the desired singularity.

\section{Open questions}
\label{sec:Open_Questions}

Here is a list of a few open questions which have come up during this survey's investigations.


\begin{question}
\label{th:prob_suspects}
Which of the new ``suspect'' pairs $(n,k)$ found in \textbf{Section \ref{sec:Computer_Results}} are singular? 
\end{question}

Of course, we are talking here about pairs highlighted in bold in Table \ref{table_nsk}. Perhaps some cleverly written computer program could answer this question at least for the smallest ``suspect'' pairs $(27,6)$ and $(32,10)$.

\begin{question}
\label{th:root_means_singular}
Does a nontrivial integer root of Moser polynomial always provide us with a singular pair? That is, is it true, for any positive integers $s$, $k$, $n$ such that $n>s$, $n\gte k$, and $\fsk sk(n) = 0$, that there exists a pair of different $n$-multisets $A$ and $B$ such that $\sms As = \sms Bs$?
\end{question}

All the results accumulated over the last sixty years so far confirm this hypothesis; however, it looks like an extremely difficult nut to crack. The following question could perhaps serve as a small step in this direction.

\begin{question}
\label{th:conjugation_singular}
A singular pair $P = (n, s)$ is such that $\fsk sk(n) = 0$ for $k=4$ or $5$. Is pair $P' = (m,s)$ obtained from $P$ by $\langle s,k\rangle$-conjugation also singular?
\end{question}

We know that $n$-conjugation between roots of $\fsk sk$ has its direct analog in the $(n,s) \leftrightarrow (n, n-s)$ duality between singular pairs. However, the similar question about $s$-conjugation does not seem to be even remotely as simple. 

\begin{question}
\label{th:math_basis}
Is there a less ``accidental'' explanation for at least one singular pair with $n\neq 2s$, $s>2$?
\end{question}

So far, all such examples were constructed in a rather \textit{ad hoc} manner by grinding through solutions for simultaneous equations $\sigma_k(\sms As) = \sigma_k(\sms Bs)$ in cases where resulting polynomials were not overwhelmingly complex. Perhaps for at least some singular pairs there exists a less ``accidental'' construction, combinatorial or algebraic. 

\smallskip

The following hypothesis seems to be the most important open question about the roots of Moser polynomials.

\begin{conjecture}
\label{th:prob_kmax} \textbf{($k_{\max}$-Conjecture)}
Is it true that $k_{\max}(s) = 2s-1$ for all $s>3$? In other words, prove or disprove that  polynomial $\fsk sk$ can have integer roots only if $k \lte 2s-1$.
\end{conjecture}

A positive answer to this question would be a considerable breakthrough in any search for the roots of Moser polynomials, computer-aided or otherwise. It would also give us solutions to various parts of Multiset Recovery Problem. 

For starters, we could immediately claim that $\mcm_s = \{2s\}$ for many small values of~$s$. For instance, the proof for $s=5$ would go like this:

\begin{proof}
If $n>5$ and $n\neq 10$, then $\fsk 5k(n) \neq 0$ for any $k>0$. Indeed, the \textbf{$k_{\max}$-Conjecture} implies there is no need to check values $k\gte 10$. From \textbf{Section \ref{sec:Roots_Fsk}} we know the same is true for $k\lte 5$. So we only need to examine $\fsk 56$, $\fsk 57$, $\fsk 58$, and $\fsk 59$:
\begin{align*}
\fsk 56(n) &= \frac{1}{24}\left(n^4 - 134 n^3 + 3311 n^2 - 27754 n + 75000\right) \pc \\
\fsk 57(n) &= \frac{1}{24}\left(n^4 - 262 n^3 + 9527 n^2 - 107570 n + 375000\right) \pc \\
\fsk 58(n) &= \frac{1}{24}\left(n^4 - 518 n^3 + 27791 n^2 - 420490 n + 1875000\right) \pc \\
\fsk 59(n) &= \frac{1}{24}\left(n^4 - 1030 n^3 + 81815 n^2 - 1653650 n + 9375000\right) \pp
\end{align*}

How do we do that? Again, we can use a computer to help us produce a verifiable computer-independent proof. As an example, let us prove (quite formally) that $\fsk 56$ has no integer roots in $[5; \infty)$. A couple of lines of code in \MLAB{} or in \SAGE{} will get us real roots of this polynomial:
$$
x_1 = 6.014875\ldots,\ x_2 = 7.745287\ldots,\ x_3 = 15.348149\ldots,\ x_4 = 104.891687\ldots \pp
$$

We cannot use that as a proof, but we can compute \textit{by hand} values $\fsk 56(x)$ for $x = 6$, 7, 8, 15, 16, 104, and 105. The results---1, (-25), 15, 85, (-199), (-31065), and 3895---prove that there is a non-inte\-ger root inside each one of intervals $[6,7]$, $[7,8]$, $[15,16]$, and $[104,105]$. Those four non-inte\-ger numbers obviously constitute the set of all roots of $\fsk 56$.

Same reasoning (but with longer computations) does it for polynomials $\fsk 57$, $\fsk 58$, and $\fsk 59$ as well ($\fsk 57$ and $\fsk 59$ both have one integer root but it is equal to 10).

An alternative way to prove the absence of nontrivial roots is to use residues modulo prime $p=13$ for both polynomials $\fsk 56$ and $\fsk 58$. Polynomial $24\fsk 57$ can be factored as $(n-10)(n^3 - 252n^2 + 7007n - 37500)$, and the absence of integer roots other than 10 can be proved using $p=23$. For $k=9$ we have $24\fsk 59 = (n-10)(n^3 - 1020n^2 + 71615n - 937500)$, and once again, residues modulo $p=13$ do the trick.

Finally, since $\fsk 5k(n)\neq 0$, then from \textbf{Theorem \ref{th:F_skn}} it follows that $n$-multiset $A$ is always recoverable from~$\sms A5$.\footnote{\ Due to the results in \cite{AmZel}, case $s=5$ is already complete. However, the readers can apply the same ideas to cases $s=6$ and $s=7$, proving, for instance, that polynomials $\fsk 78$ through $\fsk 7{13}$ have no integer roots.}
\end{proof}

Proving this conjecture would also immediately provide us with the full resolution of Multiset Recovery Problem for cases $s=7$, 8, and 9, as well as for many other values of~$s$.

\smallskip

One other hypothesis proposes an update to \textbf{Question \ref{th:prob_mtx}}.

\begin{conjecture}
\label{th:prob_mtgx} \textbf{(Magical Triplet Conjecture)}
If $s>2$ and $n>2s$, then for any three distinct $n$-multisets some two of them are not $s$-equivalent to each other.
\end{conjecture}

Let us call three different $n$-multisets such that they are all $s$-equivalent to each other, a \textit{magical triplet}. I submit that outside of cases $s=2$ and $n=2s$ (which have been already investigated quite thoroughly) \textbf{magical triplets do not exist}---in other words, when one tries to recover a multiset from its collection of $s$-sums they will always have no more than two options to choose from.

\section{Who Is Who}
\label{sec:Who_Is_Who}

The Multiset Recovery Problem, despite its elementary nature, has attracted attention of several prominent mathematicians. I would like to honor all the contributors here by listing their (very short) bios below. I apologize in advance for any factual errors and possible incompleteness of this list.

\medskip

\textbf{Leo Moser} was born in Vienna, Austria in 1921, then brought by his parents to Canada three years later. He became professor of mathematics at University of Alberta after spending a few years in the United States (Ph.D. from University of North Carolina, then post-doc at Texas Technical College). His interests lay mostly in number theory, combinatorics and combinatorial geometry. He died in Toronto at the age of 48 from heart failure.

\smallskip

\textbf{Joachim (Jim) Lambek} was born in Leipzig, Germany in 1922. As a teenager he emigrated to England, and then, as the World War II began, he was forcibly relocated to Canada. After prison work camp he went to McGill University where he earned his M.Sc in Mathematics in 1946. Then in 1950 he completed his Ph.D under Hans Zassenhaus, concentrating his research on combinatorics and number theory. However, later he developed an interest in mathematical and computational linguistics, and he worked in general algebra, pregroups and formal languages. He retired from McGill as Professor Emeritus, and died in Montreal in 2014.

\smallskip

\textbf{Ernst Gabor Straus} was born in Munich, Germany in 1922. In 1933 his family fled from the Nazis to Palestine. After World War II was over he went to Columbia University, earning his Ph.D there in 1948. In 1950 at Institude for Advanced Studies (Princeton) he became assistant to Albert Einstein. His interests ranged quite widely, from relativity to number theory, graph theory and combinatorics. Most of his career after IAS he worked at UCLA. Professor Straus died in 1983 from heart failure.

\smallskip

\textbf{John Lewis Selfridge} was born in Ketchikan, Alaska in 1927. He worked mostly in number theory and combinatorics, getting his Ph.D from UCLA, and then working in University of Illinois at Urbana-Champaign and Northern Illinois University. He was a founder of Number Theory Foundation which named one of its prizes in his honor. He died in 2010 at the age of 83.

\smallskip

\textbf{John Albert Ewell} was born in Newellton, Louisiana, in 1928. He received both his M.S. and Ph.D in Mathematics from UCLA. His Ph.D. thesis was partially summarized in the article \cite{Ewe} about the multiset recovery problem that we have mentioned above; professor Straus was his advisor. He worked in number theory, teaching at various universities in California, Canada and Illinois. Dr. Ewell died in 2007 in Knoxville, Tennessee.

\smallskip

\textbf{Aviezri Siegmund Fraenkel} was born in Munich, in 1929. He lives in Israel and his main field of research is combinatorial game theory and computational complexity. He received his Ph.D at UCLA, where Ernst Straus was his advisor. In 2005 he was awarded Euler Prize for his contributions to combinatorics.

\smallskip

\textbf{Basil Gordon} was born in Baltimore, Maryland in 1931. He received his Ph.D in Mathematics and Physics from Cal Tech under Tom Apostol and Richard Feynman. He was then drafted into the US Army where he worked with Werner von Braun. Upon returning to academia, he specialized in number theory, algebra, and combinatorics. After working many years at UCLA, Professor Gordon retired in 1993. He died in California, in 2012.

\smallskip

\textbf{Jan Boman} was born in Sweden in 1933. He obtained his Ph.D from Stockholm University with Lars Hörmander as his thesis advisor. Currently he is Professor Emeritus at Stockholm University, and his research is mainly in the areas of integral geometry and calculus, with many papers on mathematical problems related to computerized tomography and Radon transform. 

\smallskip

\textbf{Ethan Bolker} was born in Brooklyn, New York in 1938. He is a retired Professor Emeritus at University of Massachusetts (Boston). He works in computer science, combinatorics, geometry and math education. He received his Ph.D from Harvard University in 1965 under supervision of Andy Gleason. In 1972 he became a Full Professor at UMass Boston where he worked until his retirement in 2014.

\smallskip

\textbf{Patrick Eugene O'Neil} was born in 1942, in Mineola, NY. He has done his Ph.D thesis in combinatorial mathematics at Rockefeller University (New York) in 1969; Gian-Carlo Rota was his thesis advisor. His scientific interests lied in database indexing and performance research, data warehousing and other adjacent areas of computer science. In 1988 he retired as a Professor Emeritus of Computer Science at University of Massachusetts (Bos\-ton). He died in Cambridge, Massachusetts in~2019.

\smallskip

\textbf{Oleg Izhboldin} was born in Leningrad, USSR in 1963. He received his M.Sc (1985) and Ph.D (1988) in Mathematics from Leningrad (St.Petersburg) State University, where he worked in algebraic K-theory with Alexander Merkurjev as his advisor. Professor Izhboldin received his Dr.Sc.~degree in 2000; just a few months later, during his trip to Paris, he suffered massive stroke and died at the age of~37.

\smallskip

\textbf{Dmitri Fomin} was born in Leningrad in 1965. He received his M.Sc in Mathematics (1986) from Leningrad State University and then pursued Ph.D in low-dimensional and algebraic topology under Oleg Viro. His interests cover discrete mathematics, geometry and topology, history of mathematics and science, as well as various topics in mathematical problem solving. He lives in Massachusetts and works in applied mathematics and computer science.

\smallskip

\textbf{Konstantin Kokhas} was born in 1966, in Leningrad. His Ph.D (2005) was in the field of spectral operator theory. His other interests lie in representation theory and functional analysis. He is a senior lecturer at the St.Petersburg State University, Faculty of Mathematics and Mechanics. In addition, Professor Kokhas is very active in mathematical education and mathematical contests for high school students.

\smallskip

\textbf{Tewodros Amdeberhan} was born in Ethiopia. He gained his M.Sc at Addis Ababa University, and then did his Ph.D on the WZ theory and proof theory at Temple University, Philadelphia under Doron Zeilberger in 1997. He is a Senior Professor of Practice at Tulane University, with his interests lying in combinatorics, number theory, proof theory and computer algebra.

\smallskip

\textbf{Melkamu Zeleke} was born in 1968 in Ethiopia. He received B.Sc and M.Sc degrees in mathematics from Addis Ababa University, and then gained his Ph.D. in Discrete Radon Transform, Covering Congruences, and Boolean Functions at Temple University, Philadelphia under the supervision of Doron Zeilberger in 1998. He is a Professor of Mathematics at William Paterson University of New Jersey, specializing in Enumerative Combinatorics.

\smallskip

\textbf{Svante Linusson} was born in 1969 in Göteborg, Sweden. He is currently a Professor at KTH Royal Institute of Technology. His main field of research is algebraic combinatorics, as well as other areas of discrete mathematics, both pure and applied. He is also interested in mathematical aspects of electoral systems.

\smallskip

\textbf{Javlon Isomurodov} was born in 1993, in Navoiy, Uzbekistan. He has graduated the B.Sc program at ITMO University in St.Peters\-burg, Russia in 2016. His bachelor thesis on injectivity of combinatorial Radon transform was written under Konstantin Kokhas' guidance. Currently he is pursuing his M.Sc degree in the area of computational genetics at ITMO.

\end{document}